\documentclass[12pt]{amsart}
\usepackage{amsmath,amscd}
\usepackage{geometry,amsfonts,amssymb,amsthm,txfonts,pxfonts,amscd,algorithmic,algorithm} 
\usepackage{refcheck}
\norefnames
\nocitenames
\def\struckint{\mathop{%
\def\mathpalette##1##2{\mathchoice{##1\displaystyle##2}%
  {##1\textstyle##2}{##1\scriptstyle##2}{##1\scriptscriptstyle##2}}%
\mathpalette
{\vbox\bgroup\baselineskip0pt\lineskiplimit-1000pt\lineskip-1000pt
\halign\bgroup\hfill$}
{##$\hfill\cr{\intop}\cr\diagup\cr\egroup\egroup}%
}\limits}
\newtheorem{theorem}{Theorem}[section]
\newtheorem{lemma}[theorem]{Lemma}
\newtheorem{corollary}[theorem]{Corollary}

\newtheorem{definition}[theorem]{Definition}

\theoremstyle{remark}

\newtheorem{remark}[theorem]{Remark}

\newtheorem{question}[theorem]{Question}

\newcommand{\integers}{\mathbb{Z}}

\DeclareMathOperator{\tr}{tr}

\DeclareMathOperator{\SL}{SL}

\DeclareMathOperator{\Out}{Out}
\DeclareMathOperator{\Inn}{Inn}
\DeclareMathOperator{\Aut}{Aut}

\providecommand{\nch}[1]{\left\langle \left\langle #1 \right\rangle\right\rangle}

\begin{document}
\title{Rigidity of fibering}
\author{Igor Rivin}
\address{Department of Mathematics, Temple University, Philadelphia}
\address{School of Mathematics, Institute for Advanced Study, Princeton}
\email{rivin@temple.edu}
\thanks{I would like to thank Ilya Kapovich for much enlightenment on geometric group theory, D. Arapura for supplying the proof of Theorem \ref{bsa}, Igor Belegradek and Alain Valette on putting me abreast of the recent developments on groups with finite outer automorphism group, Fred Cohen for interesting discussions on different sorts of monodromy, N. Katz and P. Sarnak, who piqued the author's interest in monodromy, and the Institute for Advanced Study for its hospitality, which made this work possible. I would like to thank Igor Belegradek, Tom Church, and David Futer on comments on a previous version of this preprint -- in particular, Belegradek had pointed out the very interesting preprint \cite{churchfarbthib}. I would also like to thank Jim Bryan, Ron Donagi, Andras Stipsicz, and Benson Farb for interesting discussions of the first version of this paper}
\begin{abstract}
Given a manifold $M,$ it is natural to ask in how many ways it fibers
 (we mean fibering in a general way, where the base might be an orbifold -- this could be described as \emph{Seifert fibering})There are group-theoretic obstructions to the existence of even one fibering, and in some cases (such as K\"ahler manifolds or three-dimensional manifolds) the question reduces to a group-theoretic question. In this note we summarize the author's state of knowledge of the subject.
\end{abstract}
\maketitle
\tableofcontents
\section{Introduction}
This note is inspired by the following celebrated theorem of A.~Beauville (\cite{beauville}) and Y.-T.~Siu (\cite{siu87}):
\begin{theorem}[Beauville-Siu]
\label{beauvillesiu}
Let $X$ be a compact K\"ahler manifold and $g \geq 2$ an integer. The $X$ admits a non-constant holomorphic map to some compact Riemann surface of genus $g^\prime \geq g$ if and only if there is a surjective homomorpjsm $h:\pi_1(X) \rightarrow \pi_1(C_g),$ where $\pi_1(C_g)$ is the fundamental group of a compact Riemann surface of genus $g.$
\end{theorem}
This results was shown to hold by D.~Kotschick for compact complex surfaces -- see \cite{compactkahler}[Theorem 2.17]
In addition, there is the following ``orbifold'' version  of Theorem \ref{beauvillesiu}:
\begin{theorem}[D. Arapura, \cite{araflow}]
\label{bsa}
Let $X$ be a compact K\"ahler manifold, then $\pi_1(X)$ admits a surjective map on a fundamental group of a compact \emph{hyperbolic} orbifold if and only if $X$ admits a non-constant holomorphic  map to an orbifold of negative Euler characteristic.
\end{theorem}
These results can be viewed as analogous (in the complex/K\"ahler category) to the classical \emph{Stallings fibration Theorem}, which states that:
\begin{theorem}[Stallings Fibration Theorem, \cite{stallingsfib}]
\label{stallingsthm}
A compact irreducible 3-manifold $M^3$ fibers over $S^1$ if and only if $\pi_1(M^3)$ admits a surjection onto $\mathbb{Z}$ with finitely generated kernel.
\end{theorem}
Theorems \ref{beauvillesiu}, \ref{bsa}, \ref{stallingsthm} raise the questions of \emph{how many} fibrations there are:
\begin{question}
\label{kahlerq}
Let $G = \pi_1(X).$ How many surjective maps $h:G\rightarrow S$ are there, such that $\ker(h)$ is the fundamental group of a compact K\"ahler manifold, and $S$ is the fundamental group of a compact $2$-dimensional orbifold?
\end{question}
The general group-theoretic form of Question \ref{kahlerq} is:
\begin{question}
\label{groupq}
Let $G$ be a group, and $\mathcal{C}_1, \mathcal{C}_2$ are two classes of groups. How many extensions 
\begin{equation}
\label{extensiondiag}
1\rightarrow K \rightarrow G \rightarrow B \rightarrow 1
\end{equation}
are there, with $K \in \mathcal{C}_1$ and $B \in \mathcal{C}_2?$
\end{question}

It should be noted that this question was considered for torsion-free Fuchsian groups by F. E. A. Johnson, who showed the following result in the very nice paper \cite{johnsonrigid}:
\begin{theorem}[F. E. A. Johnson, \cite{johnsonrigid}]
For any group $G,$ there is a finite number of extensions of type \eqref{extensiondiag} with $K, B$ torsion-free fuchsian groups.
\end{theorem}

Johnson points out in his paper that an algebraic-geometric form of  his finiteness result was a celebrated result of A. N. Parshin \cite{parshin}.

In this paper we will consider some aspects of Questions  \ref{kahlerq} and \ref{groupq}. In Section \ref{khs} we start looking at fiberings of K\"ahler surfaces (which are manifolds of real dimension $4.$). First, we observe that (in the setting of Extension \eqref{extensiondiag}) if the group $G$ has cohomological dimension at most $2$ there is no extension with $\mathcal{C}_1$ the class of K\"ahler groups (fundamental groups of compact K\"ahler manifolds) and $\mathcal{C}_2$ the class of fundamental groups of compact orbifolds of negative Euler characteristic (Theorem \ref{bierithm}) (so if a K\"ahler manifold had such fundamental group, it would not fiber with base hyperbolic surface (or orbifold, which corresponds to a "Seifert fibering", with some exceptional fibers).

Then we consider a situation where a group $G$ has an extension \eqref{extensiondiag} with $\mathcal{C}_1$ and $\mathcal{C}_2$ nonelementary Fuchsian groups. We show that there are arithmetic obstructions to multiple fiberings with prescribed ranks of the fiber and base groups (Theorem \ref{uniquefact}, proved in Section \ref{uniqueproof}).

The arguments used to prove Theorem \ref{uniquefact} do not work when the fibers are tori, but it turns out that if the fiber group in Extension \eqref{extensiondiag} is \emph{solvable} (in particular, abelian), very strong rigidity results can be shown, thanks to recent advances in geometric group theory -- this is the subject of Section \ref{solfib}. These results are applied to the question of fibering K\"ahler surfaces in Corollary \ref{ellipticuniq}.

Finally, the structure of Extension \eqref{extensiondiag} is determined, in many important cases, by the \emph{monodromy representation} of $B$ into the outer automorphism group $\Out(N)$ of $N.$
This is discussed in Section \ref{monosec}. It turns out that in many interesting cases, in an extension of the form \eqref{extensiondiag} is virtually a direct product.
\section{K\"ahler surfaces}
\label{khs}
\subsection{Negative results}
There are certainly cases where no such map exists for group-theoretic reasons. 
\begin{theorem}
Let $G$ be a a group with Kazhdan's property T (or, for example, property $\tau$ with respect to representations with finite image). Then there is no surjective map $G\rightarrow H\rightarrow 1,$ where $H$ does \emph{not} have property T (or property $\tau,$ respectively).
\end{theorem}
\begin{proof} Immediate from the definition of property T (or property $\tau$): a "bad" sequence of representations of $H$ gives a "bad" sequence of representations for $G.$
\end{proof}
\begin{remark}
An excellent reference for Property T is \cite{bekkadlhvalette}, for Property $\tau,$ see, for example, \cite{alextau}.
\end{remark}
\begin{corollary}
There is no exact sequence $1\rightarrow N \rightarrow G \rightarrow H \rightarrow 1,$ where $G$ has property T or property $\tau$ for representations with finite image, and $H$ is an infinite Fuchsian group.
\end{corollary}
\begin{proof} This follows from the well-known fact that infinite Fuchsian groups \emph{do not} have property T, nor property $\tau$ for representations with finite image (in particular because they surject onto $\integers.$)
\end{proof}
\begin{corollary}
No K\"ahler manifold whose fundamental group has property T or property $\tau$ for representations with finite image fibers over a hyperbolic or Euclidean two-dimensional orbifold.
\end{corollary}
For another sort of an obstruction, recall Bieri's Theorem:
\begin{theorem}[R.~Bieri, \cite{bierinotes}]
\label{bierithm}
If $G$ is a group of cohomological dimension of at most two, while $N$ is a normal subgroup of $G$ of infinite index, then either $N$ is free, or $N$ is not finitely presentable.
\end{theorem}
From this we have the following corollary:
\begin{corollary}
Let $M^n$ be a  manifold whose fundamental group is of cohomological dimension at most $2.$ Then $M$ is not a K\"ahler manifold which fibers over a compact two (real) dimensional base with compact fiber.
\end{corollary}
\begin{proof}
If such a fibration existed, the fundamental group of the fiber would be of infinite index, hence either free or not finitely presentable. In the latter case, it is not the fundamental group of a compact manifold, in the former case, it is not the fundamental group of a compact \emph{K\"ahler} manifold, by the results of D. Arapura, P. Bressler, and M. Ramachandran \cite{arabrer}.
\end{proof}
\begin{remark}
The class of groups of cohomological dimension at most two is quite large, since it includes free groups and fundamental groups of compact surfaces, and is closed under the operations of taking free products with amalgamation and HNN extensions, where the amalgamated or associated products are free, or, more generally, graphs of groups with edge groups free. 

In addition, any one-relator group $G= \langle a_1,..., a_k \left| r=1\right.\rangle$ where $r\in F(a_1,.., a_k)$ is a cyclically reduced word which is not a proper power, not a primitive element in $F(a_1,..., a_k)$ and involves all of $a_1,..., a_k,$ then $G$ is torsion free, freely indecomposable (and thus one-ended) and has cohomological dimension two.\end{remark}
\begin{question} Are there complex surfaces other than $P^1 \times C,$ where $C$ is any curve, whose fundamental groups are of cohomological dimension at most two?
\end{question}
\begin{remark} The "trivial exceptions" were pointed out by Tom Church.
\end{remark}
\subsection{Obstructions to multiple fibering}
It is clear that there are cases when there is more than one fibration, in particular, when $X = C_g \times C_h,$ it can be viewed as a fibration with either factor as the base. Our main result is the following theorem:

\begin{theorem}
\label{uniquefact}
Suppose
\begin{gather*}
1\rightarrow N \xrightarrow{i} G \xrightarrow{\psi} S_2 \rightarrow 1\\
1\rightarrow K \xrightarrow{i} G \xrightarrow{\phi} S_1 \rightarrow 1,
\end{gather*}
where $N,K, S_1, S_2$ are nontrivial finitely generated fundamental groups of compact hyperbolic surfaces.  Then we have the following possibilities:
\begin{enumerate}
\item $N = K$ and $S_1 \simeq S_2.$
\label{trivcase}
\item 
\label{dprod}
The genus of $N$ equals the genus of $S_1$ and the genus of $K$ equals the genus of $S_2.$ In this case $N \simeq S_1,$ $K\simeq S_2,$ and $G = N \times K.$
\item $G$ is a nontrival finite extension of the direct product of $N$ and $K.$ In this case 
\[
\dfrac{g(N)-1}{g(S_1)-1} = \dfrac{g(K)-1}{g(S_2)-1} =q > 1,
\] where $q$ is the order of the extension (and so an integer).
\label{fprod}
\item  The group $G$ is a finite extension of $N K,$ while $N \cap K$ is an infinitely generated free group. In this case the genus of $N$ is greater than that of $S_1,$ while the genus of $K$ is greater than that of $S_2.$ In this case, it is also true that
\begin{equation}
\label{johnsoneq}
(g(N)-1)(g(S_2)-1) = (g(K)-1)(g(S_1) -1).
\end{equation}
\label{hardcase}
\end{enumerate}
\end{theorem}

\begin{remark}
\label{uniquefact2}
The hypothesis that $N, K, S_1, S_2$ are finitely generated fundamental groups of \emph{compact} hyperbolic surfaces is excessive. The result  (with the possible exception of Eq. \eqref{johnsoneq})holds if the groups are any finitely generated fuchsian groups, though Case \eqref{hardcase} needs to be modified slightly: if all groups are free, we simply replace genus by the rank. if $N, K$ are free and $S_1, S_2$ are compact surface groups (or vice versa), we replace genus by rank, again, but now we can replace "greater" by "greater than or equal to".
\end{remark}
\begin{remark} 
\label{nosig}
By Corollary \ref{kotcor} below, if we think of the group extensions as corresponding to surface-over-surface fibrations, then, by Corollary \ref{kotcor}  the signature of the total space vanishes in Cases \ref{dprod} and \ref{fprod}.
\end{remark}
\begin{remark}
As pointed out in \cite{churchfarbthib} -- the result seems to be actually due to \cite{bryandonagistip}, where the authors found the second fibering, -- the Atiyah-Kodaira manifold $N^4$ fibers in two ways, thusly:
\begin{gather}
\Sigma_4\rightarrow N^4 \rightarrow \Sigma_{17},\\
\Sigma_{49}\rightarrow N^4 \rightarrow \Sigma_2,
\end{gather}
where $\Sigma_g$ is a surface of genus $g.$
Notice that in this case, using the notation of Theorem \ref{uniquefact},
$g(N) = 4, g(S_2) = 17, g(K) = 49, g(S_1) = 2,$ so
\[
3=\dfrac{g(N)-1}{g(S_1)-1} = \dfrac{g(K)-1}{g(S_2)-1},
\]
which falls into the situation of Case \ref{fprod} of Theorem \ref{uniquefact}, so it is reasonable to conjecture that $N^4$ is triply covered by $\Sigma_4 \times \Sigma_{49}.$ Since, however, the signature of $N^4$ equals $32 \neq 0,$ this is not possible by Remark \ref{nosig}

A general construction of $4$-manifolds which fiber over a surface in different ways was given by Jim Bryan and Ron Donagi in \cite{bryandonagi}. They define a family of complex surfaces $X_{g, n},$  where $X_{g, n}$ admits two fiberings. In the language of Theorem \ref{uniquefact}, for $X_{g, n}$ we have:
\begin{gather*}
g(S_2) =  g,\\
g(N) =  g(gn -1)n^{2g - 2} + 1,\\
g(S_1) = g(g-1) n^{2g - 2} + 1,\\
g(K) = g n.
\end{gather*}
We then have for the Bryan-Donagi surfaces:
\[
\dfrac{g(N)-1}{g(S_1)-1} = \dfrac{g(K) -1}{g(S_2)-1} = \dfrac{gn-1}{g-1},
\]
so the two ratios still agree, although they are not integral for any value of $g, n, $ so we know we are in Case \ref{hardcase} of Theorem \ref{uniquefact} even without bringing in the signature  -- Bryan and Donagi show that:
\[
\sigma(X_{g, n} )= \frac43 g(g-1)(n^2-1) n^{2g - 3}.
\]

T. Church and B. Farb gave a general construction of manifolds with multiple fiberings (with fiber a surface) in \cite{churchfarbthib}[Section 3].
\end{remark}
\begin{remark}
A deep study of the situation of Case \ref{fprod} of Theorem \ref{uniquefact} has been undertaken by F. Catanese and coauthors in papers \cite{cat1,cat2,cat3,cat4}
\end{remark}
\section{Proof of Theorem \ref{uniquefact}}
\label{uniqueproof}
To prove Theorem \ref{uniquefact} we will need the following results:

\begin{theorem}
\label{bieri}
A finitely generated nontrivial  \emph{normal} subgroup of either a free group or a fundamental group of a compact surface is of finite index.
\end{theorem}
The proof of this theorem is contained in the first paragraph of \cite{KarrassSolitar1R}.

\begin{corollary}
\label{bieri2}
A finitely generated \emph{infinite} normal subgroup $N$  of a finitely generated fuchsian group $G$ is of finite index.
\end{corollary}
\begin{proof}
Let $N$ be as in the statement of the corollary. By Selberg's lemma, $G$ has a torsion free subgroup $H$ of finite index, and $N\cap H$ is a normal subgroup of $H.$ Since $H$ is  of finite index, it is finitely generated, and since $N \cap H$ is of finite index in $N,$ \emph{it} is finitely generated, hence of finite index in $H,$ hence of finite index in $G.$
\end{proof}

In fact, in the cases of interest, Corollary \ref{bieri2} gives the same result as Theorem \ref{bieri}, because of the following observation:

\begin{lemma}
\label{bieri3}
An infinite discrete group of isometries of an Hadamard manifold $H^n$has no finite normal subgroup.
\end{lemma}
\begin{proof}
Let $G$ be our  group, which is the fundamental group of some orbifold $S_G.$ Let $N$ be the finite normal subgroup, and let $S_N$ be the corresponding regular covering space of $S_G.$ Since $N$ is finite, its action on $H$ has a fixed point $p,$ which is an orbifold point of $S_N.$ The space $S_G$ is a quotient of $S_N$ by a group of isometries $H=G/N$ (since the covering is regular), but that group must fix $p,$ and furthermore, it must fix the distance to $p,$ which means that $H$ is a (discrete) subgroup of $SO(n),$ hence finite.
\end{proof}
\begin{theorem}
\label{jaco}
A finite index subgroup of the fundamental group of a compact orientable surface of genus at least $2$ is itself the fundamental group of a hyperbolic surface. An \emph{infinite} index subgroup of the fundamental group of a compact hyperbolic surface is free.
\end{theorem}

The proof of this theorem is given in \cite{jacosubgroups}. More general results for groups with torsion were obtained by A. Hoare, A. Karrass, and D. Solitar in a pair of nice papers \cite{hksfinite,hksinfinite}, using completely different (and purely combinatorial) methods. Their results for groups with torsion read as follows:
\begin{theorem}[\cite{hksfinite}]
\label{finiteorbi}
Let $F$ be the fundamental group of a compact hyperbolic orbifold. The group $F$ is given by the presentation
\begin{equation}
\label{fuchspres}
F=\langle a_1, b_1, \dotsc, b_n, b_n, c_1, \dotsc, c_t; c_1^{\gamma_1}, \dotsc, c_t^{\gamma_t}, c_1^{-1} \dots c_t^{-1} [a_1, b_1] \dots [a_n, b_n]\rangle.
\end{equation}
Then any finite index subgroup has a presentation of the same type
\end{theorem}
\begin{remark} Theorem \ref{finiteorbi} has an immediate topological proof, as had already been observed by Fricke and Klein in \cite{frickeklein}. Hoare, Karrass, and Solitar's proof is completely combinatorial.
\end{remark}

\begin{theorem}[\cite{hksinfinite}]
\label{infiniteorbi}
A subgroup of infinite index of the fundamental group of a (not necessarily compact) hyperbolic orbifold is a free product of cyclic groups.
\end{theorem}

\begin{theorem}
\label{homeo}
Let $S_g$ and $S_h$ be compact hyperbolic surfaces of genus $g$ and $h$ respectively. There is a surjection from $\pi_1(S_g)$ to $\pi_1(S_h)$ if and only if $g\geq h.$ If $g > h,$ the kernel of such a map is an infinitely generated free group, if $g=h$ then the map is an isomorphism.
\end{theorem}
\begin{proof}
This follows immediately from Theorem \ref{jaco}.
\end{proof}

\begin{theorem}[9-Lemma]
\label{ninelemma}
Consider two homomorphisms from a group $G$ to groups $S_1$ and $S_2,$ thus:
\begin{gather*}
1\rightarrow N \xrightarrow{i} G \xrightarrow{\psi} S_2 \rightarrow 1\\
1\rightarrow M \xrightarrow{i} G \xrightarrow{\phi} S_1 \rightarrow 1,
\end{gather*}
where $i$ represents the inclusion map. Then, the following diagram (where all the maps are natural) commutes.

\[
\begin{CD}
@. 1 @. 1 @. 1 @. \\
@. @VVV @VVV @VVV @.\\
1 @>>>{K\cap N} @>i>> N @>{\phi}>> {M =N/(K\cap N)=KN/K}@>>> 1\\
@. @ViVV @ViVV @ViVV @. \\
1 @>>> K @>i>> G @>{\phi}>> {S_1} @>>> 1\\
@. @V{\psi}VV @V{\psi}VV @V{\overline{\psi}}VV @. \\
1 @>>> R={K/(K\cap N) = KN/N} @>i>> S_2 @>{\overline{\phi}}>> Q @>>> 1\\
@. @VVV @VVV @VVV @. \\
@. 1 @. 1 @. 1 @. 
\end{CD}
\]
\end{theorem}
\begin{proof}
This is just the nine-lemma of homological algebra.
\end{proof}
\begin{remark} We will use the notation in the diagram in Lemma \ref{ninelemma}
\end{remark}

Consider now a pair of homomorphisms given by the two exact sequences in the statement of Theorem \ref{uniquefact}, and fit them into a diagram as in the statement of Theorem \ref{ninelemma}.

By Theorem \ref{bieri}, the group $M$ (which is a normal subgroup of $S_1$, and finitely generated as the image of the finitely generated group $N$) is either trivial or of finite index in $S_1$.
\subsection*{Case 1: $M = \{1\}$} It follows that $K\cap N = N,$ so $N \subseteq K.$ It also follows that $Q\simeq  S_1.$ Since $K/(K\cap N)$ is a finitely generated normal subgroup of $S_2,$ it must be trivial by Theorem \ref{bieri}, which means that $K = N,$ and therefore the two homomorphisms differ by post-composition with an automorphism of $S_1(\simeq S_2).$

\subsection*{Case 2: $M$ is of finite index in $S_1.$} This means that $KN/N$ is of finite index in $S_2,$ which means that $KN$ is of finite index in $G.$ Now, if $N$ is isomorphic to $M,$ it follows that $K\cap N = \{1\},$ and so $G$ is a finite extension of a direct product of surface groups. If the extension is nontrivial, then $N$ is isomorphic to a proper finite subgroup of $S_1$ and so $(g(S_1) -1) |Q| = g(M)-1,$ and also $(g(S_2)-1)|Q| = g(K) -1,$ by the same argument.

 If $N$ is \emph{not} isomorphic to $M,$ things are more complicated.  If $N, K, S_1, S_2$ are surface groups, we know that $g(N)> g(M) \geq g(S_1).$ Also, $K/(K\cap N)$ is of finite index in $S_2,$ which means that $g(K)>g(K/(K\cap N)) \geq g(S_2).$ That the relationship Eq. \eqref{johnsoneq} holds when all the groups are fundamental groups of compact surfaces was shown by F. E. A. Johnson in \cite{johnsonrigid} -- the proof is based on the spectral sequence argument, which shows that in the extension $1\rightarrow N \rightarrow G \rightarrow B \rightarrow 1,$ when $N, B$ are surface groups, the Euler characteristic of $G$ is defined and is equal to the product of the Euler characteristics of $N$ and of $B.$ In the "topological case", where $G = \pi_1(M^4),$ where $M^4$ is a surface bundle over a surfaces, we know from the long exact sequence of the fibration that $M^4$ is aspherical, and so the Euler characteristic of $G$ is equal to the Euler characteristic of $M^4.$

\begin{definition}
We call the smallest normal subgroup in a group $G$ containing an element $x$ by $\nch{x},$ and similarly the smallest normal subgroup containing a set $S$ of elements by $\nch{S}.$
\end{definition}

\begin{lemma}
\label{normalclosure}
 Let $G$ be a fuchsian group containing at least one hyperbolic element. Then $\nch{x}$ contains an infinite cyclic subgroup for any element $x \in G.$
\end{lemma}
\begin{proof}
The only time when the statement has content is when $x$ is an elliptic element (otherwise, $x$ itself is of infinite order). Let $\gamma \in G$ be a hyperbolic element, and denote 
$x_k = \gamma^k x \gamma^{-k},$ where $x = x_0.$ We can pick a basis where $\gamma$ is diagonal, so that $\gamma = \bigl(
\begin{smallmatrix}
\lambda & 0 \\ 0 & \lambda^{-1}
\end{smallmatrix}
\bigr)
.$
In the same basis, $x = \bigl(
\begin{smallmatrix}
a & b\\c&d
\end{smallmatrix}
\bigr),$
and a computation shows that $x_k =\bigl(
\begin{smallmatrix}
\lambda^k a & \lambda^k b\\\lambda^{-k}c & \lambda^{-k}d
\end{smallmatrix}
\bigr).$
It follows that $\tr x_k x = a^2 + d^2 ( \lambda^k+\lambda^{-k})(b+c).$

If $b+c\neq 0,$ then $\lim_{k\rightarrow \infty}|\tr x_k x|=\infty,$ and so $x_k x$ will be eventually a hyperbolic element. The only way this argument could fail is if $b+c = 0,$ but then we can try applying the argument to some $x_m$ in place of $x.$ The only way it could fail there is if $\lambda^m b + \lambda^{-m} c = 0.$ The only way this could happen for \emph{all} $m,$ is if $b=c=0,$ but that means that $x$ has the same fixed points as $\gamma,$ which contradicts the assumption that $x$ was elliptic.
\end{proof}

\begin{remark}
There are numerous other possible proofs of Lemma \ref{normalclosure}, but the one we give is completely elementary.
\end{remark}
\begin{lemma} Assume $G$ \emph{is} a direct product of  fuchsian groups $N$ and $K,$ and assume that $L$ is a normal word-hyperbolic subgroup of $G.$ Then, either $L$ is a normal subgroup of $N$ or a normal subgroup of $K.$
\end{lemma}

\begin{proof}
We will write an element $g \in G$ as $g = (n, k),$ with $n \in N, k\in K.$ If the conclusion of the Lemma does not hold, then there exists $l \in L,$ with $l = (n_l, k_l),$ with neither $n_l$ nor $k_l$ the identity.  It is easy to see that ${\nch l}_G = {\nch n_l}_N \times {\nch k_l}_K.$ By Lemma \ref{normalclosure}, it follows that ${\nch l}_G$ contains a direct product of infinite cyclic subgroups, hence is not word hyperbolic.
\end{proof}

\section{Solvable fiber}
\label{solfib}
\begin{definition}
We call a group $G$ \textbf{normally insoluble} if $G$ has no  nontrivial solvable normal subgroups.
\end{definition}
\begin{lemma}
\label{solvablethm}
Suppose 
\begin{gather*}
1\rightarrow N \xrightarrow{i} G \xrightarrow{\psi} S_2 \rightarrow 1\\
1\rightarrow K \xrightarrow{i} G \xrightarrow{\phi} S_1 \rightarrow 1,
\end{gather*}
where $N,K$ are solvable groups, and $S_1, S_2$ are normally insoluble groups. Then $N = K.$
\end{lemma}
\begin{proof}
Using the diagram in the statement of Lemma \ref{ninelemma}, we note that the group $M$ is a solvable normal subgroup of $S_1$ and,  by hypothesis trivial, which implies that $N \subseteq K.$ Similarly, $K\subseteq N,$ and the result follows.
\end{proof}
\begin{lemma}
\label{solvablethm2}
There is no pair of extensions
\begin{gather*}
1\rightarrow N \xrightarrow{i} G \xrightarrow{\psi} S_2 \rightarrow 1\\
1\rightarrow K \xrightarrow{i} G \xrightarrow{\phi} S_1 \rightarrow 1,
\end{gather*}
where $N$ is solvable and $K$ and $S_1$ normally insoluble.\end{lemma}
\begin{proof} We again use the diagram in the statement of Lemma \ref{ninelemma}.
Since $N \cap K$ is a solvable normal subgroup of $K,$ it must be that $N \cap K = \{1\}.$ It follows that $M \simeq N,$ which contradicts the normal insolubility of $S_1.$
\end{proof}
Lemmas \ref{solvablethm}, \ref{solvablethm2} are not very profound, but have very nice corollaries (thanks to very deep work of a number of authors).
\begin{theorem}
\label{trivnonsol}
A nonabelian simple group is normally insoluble.
\end{theorem}
\begin{proof}
Immediate.
\end{proof}
\begin{theorem}
\label{freeprod}
A nontrivial free product $G=G_1 * G_2$ is normally insoluble, as long as $G_1$ and $G_2$ are not both cyclic groups of order $2.$
\end{theorem}
\begin{proof}
We use the standard Bass-Serre machinery of groups acting on trees -- see \cite{serretrees}. By the Kurosh subgroup theorem, a subgroup $H$ of $G$ has the form:
\[
H =\left( *_i g_i^{-1} G_1 g_i\right) * \left(*_j g_j^{-1} G_2 g_j\right) * F(X),
\]
where $F(X)$ is a free subgroup generated by a subgroup $X = \{x_\alpha\} \subset G,$ where each $x_\alpha$ acts hyperbolically.

If $H$ is a solvable subgroup it cannot be a nontrivial free product. Since it is normal, it must have the form $<x>,$ for some hyperbolic element $x.$ But since conjugating $x$ changes its axis, $\nch{x}$ cannot be a cyclic subgroup.
\end{proof}
\begin{remark}
\label{genfreeprod}
Analogous theorems can be proved (by the same method) for amalgamated free products and general graphs of groups.
\end{remark}
\begin{theorem}
\label{tfreehyp}
Torsion-free nonelementary hyperbolic groups are normally insoluble.
\end{theorem}
\begin{proof}
It was originally noted by M.~Gromov in \cite{gromovgroups} (and proved carefully by a number of authors, see, eg, \cite{ghysdelaharpe,cdp}) that:

A subgroup of a torsion-free word-hyperbolic group is either trivial, or virtually cyclic, or contains a free subgroup on two generators.

In a torsion free word hyperbolic group $G$ , the normalizer in $G$ of any cyclic subgroup $Z$ is a finite extension of $Z.$

We are now done: a solvable subgroup of a torsion free  word-hyperbolic group $G$ must be cyclic, so its normalizer is a finite extension of $Z,$ so not all of $G,$ under our hypothesis that $G$ is nonelementary.
\end{proof}

\begin{theorem}
\label{isomh}
Nonelementary groups of isometries of Hadamard manifolds are normally insoluble.
\end{theorem}
\begin{proof}
By Lemma \ref{bieri3}, such a group $G$ has no finite normal subgroup. Now, $G$ is word-hyperbolic.  I is known (see, again \cite{gromovgroups,ghysdelaharpe,cdp}) that a subgroup $H$ of a word-hyperbolic $G$ is either
\begin{enumerate}
\item finite
\item virtually infinite cyclic
\item contains a free subgroup of rank 2.
\end{enumerate}
Further, if $Z\in G$ is virtually infinitely cyclic, then its normalizer in $G$ is a finite extension of $Z.$ Since our group was assumed nonelementary, we are done.
\end{proof}
\begin{corollary}
\label{seifert}
A Seifert fibration of a three-manifold whose base is of negative Euler characteristic is uniquely determined by its fundamental group.
\end{corollary}
\begin{proof}
Such a Seifert fiber space can be viewed as a fibration whose fiber is $S^1,$ whose fundamental group is $\mathbb{Z},$ so solvable, and the base is a hyperbolic orbifold.
\end{proof}
\begin{corollary}
\label{ellipticuniq}
Suppose a $4$-manifold $M^4$ Seifert-fibers over a hyperbolic orbifold with elliptic (torus) fiber. Then, this fibering is unique.
\end{corollary}
\begin{proof}
The fundamental group of the torus $T^n$  is $\mathbb{Z}^n,$ which is abelian, hence solvable. Lemma \ref{solvablethm2} shows that $M^4$ cannot fiber with a higher-genus fiber, while Lemma \ref{solvablethm} shows that $M^4$ admits a unique fibration with an elliptic fiber. Finally, if $M^4$ fibered with a rational (sphere) fiber, its fundamental group would be the same as that of its base, but we know that a (nonelementary) fuchsian group is normally insoluble.
\end{proof}
\begin{remark}
The hypothesis that the base is a \emph{hyperbolic} orbifold is necessary: there are K3 surfaces which admit multiple fiberings -- their base is a \emph{Euclidean} orbifold. These examples were pointed out to the author by E. Bombieri ( \cite{bombieri})
\end{remark}
\begin{theorem}
\label{margulislat}
Centerless irreducible lattices in connected semi-simple Lie groups of real rank at least two are normally insoluble.
\end{theorem}
\begin{proof} Margulis' normal subgroup theorem (\cite{margulisnorm}) states that every normal subgroup in such a lattice is of finite index, hence itself a lattice, hence (by the Tits alternative) not solvable. Recall that the Tits alternative (J. Tits \cite{titsalt}) is that a finitely generated  matrix group is either solvable-by-finite or contains a nonabelian free subgroup).
\end{proof}

\begin{theorem}
\label{longmcg}
Mapping class groups of closed surfaces of genus at least three are normally insoluble.
\end{theorem}
\begin{proof}
D. Long shows in \cite{longnorm}[Lemma 2.6] that a normal subgroup $N$ of the mapping class group of a closed surface of genus at least three (the genus hypothesis comes in to show that the mapping class group in question has trivial center) has at least two non-commuting pseudo-anosov mapping classes. Once we know that, a standard ping-pong argument (see, e.g., \cite{mcgpingpong}) shows that $N$ contains a free group on two generators, and so is not solvable.
\end{proof}

\begin{theorem}
\label{mosherout}
Outer automorphism groups of nonabelian free groups of rank greater than two are normally insoluble.
\end{theorem}
\begin{proof}

We need three results:

\begin{enumerate}
\item Culler's theorem \cite{cullerfinite}: this states that a finite subgroup of $\Out(F_k)$ stabilizes a point in ``Outer space'' of $F_k$ -- this is the analogous result to the Nielsen realization problem for $\Out(F_k).$
\item The results of M. Bestvina, M. Feighn, and M. Handel (\cite{mbhsolv}) and E. Alibegovic (\cite{alibegsolv}), which state that solvable subgroups of $\Out(F_n)$ are finitely generated virtually abelian.
\item The result of M. Feighn and M. Handel \cite{fh}, which states that an abelian subgroup of $\Out(F_k)$ is virtually cyclic.
\end{enumerate}

From this list, we deduce that a normal solvable subgroup of $\Out(F_k)$ must be an infinite virtually cyclic subgroup, which then contains a fully irreducible element (an iwip). Conjugating this element with another fully irreducible element produces a fully irreducible element with a different axis, and using standard ping-pong arguments we see that the normal closure of our subgroup contains a nonabelian free group, so is not solvable.
\end{proof}

\begin{remark}
The group $\Out(F_2)$ is isomorphic to $\SL(2, \mathbb{Z}),$ and so has nontrivial center.
\end{remark}

\section{Circle bundles over three manifolds}
\label{threem}
A slight variation on the results of Section \ref{solfb} gives us the following result:
\begin{theorem}
\label{nonseifert}
Let \begin{gather}
1 \rightarrow \mathbb{Z} \rightarrow G \rightarrow M \rightarrow 1\\
1 \rightarrow \mathbb{Z} \rightarrow G \rightarrow N \rightarrow 1
\end{gather}
be two extensions, with $M, N$ fundamental groups of  compact, orientable, non-Seifert fibered three-manifolds. Then $M= N.$
\end{theorem}
\begin{proof}
By the solution to the Seifert Fiber Space Conjecture (see \cite{cassjung,gabaiseif}) the fundamental groups of the manifolds $M^3, N^3$ have no normal cyclic subgroups, and so the proof of Theorem \ref{uniquefact} immediately implies the result (since the groups $M$ and $R$ in the diagram in the statement of Lemma \ref{ninelemma} are trivial).
\end{proof}
\section{Monodromy}
\label{monosec}
Given an exact sequence of groups 
\begin{equation}
\label{extension}
1\rightarrow N \rightarrow G \rightarrow B\rightarrow 1,
\end{equation} there is a diagram as follows:
\begin{equation}
\begin{CD}
\label{monocd}
@. 1 @. 1 @. 1 @. \\
@. @VVV @VVV @VVV @.\\
1 @>>>{Z(N)} @>i>> {Z_G(N)}@>{\phi}>> P={Z_G(N)/Z(N)=(N Z_G(N))/N}@>>> 1\\
@. @ViVV @ViVV @ViVV @. \\
1 @>>> N @>i>> G @>{\phi}>> {B} @>>> 1\\
@. @VVV @VVV @VVV @. \\
1 @>>> \Inn(N)  @>i>> {\mathcal{G} < \Aut(N)} @>{\overline{\phi}}>> {\mathcal{M} <\Out(N)} @>>> 1\\
@. @VVV @VVV @VVV @. \\
@. 1 @. 1 @. 1 @. 
\end{CD}
\end{equation}
where $Z(N)$ is the center of $N,$ and $Z_G(N)$ is the centralizer of $N$ in $G.$ The group $G$ acts on the normal subgroup $N$ by conjugation, and $\mathcal{G}$ denotes  the induced subgroup of  $\Aut(N).$ The group $N$ acts on itself by conjugation, and thus gives rise to the subgroup $\Inn(N)<\Aut(N)$ of \emph{inner} automorphisms. The quotient group $\mathcal{M} = \mathcal{G}/\Inn(N)$ is the \emph{monodromy} group of the extension \eqref{extension}.
\begin{remark}
In algebraic geometry, the monodromy representation is often an action on some module (often the cohomology of the space corresponding to $N$), so the situation considered in this section is, in some sense, complementary. 
\end{remark}
\begin{remark}
If the monodromy representation lifts to $\Aut(N),$ then the extension \eqref{extension} is split (that is, there is a subgroup $M < G,$ such that $\phi(M) = B,$ and $\phi$ is an isomoprhism restricted to $B.$) Such a lift obviously always exists if $B$ is a free group.
\end{remark}
\begin{lemma}
\label{monoquotient}
The monodromy group $\mathcal{M}$ is isomorphic to $G/(N Z_G(N)).$
\end{lemma}
\begin{proof}
From diagram \eqref{monocd} we deduce another nine-lemma diagram:
\[
\begin{CD}
@. 1 @. 1 @. 1 @. \\
@. @VVV @VVV @VVV @.\\
1 @>>>N @>i>>{N Z_G(N)}@>{\phi}>> P@>>> 1\\
@. @ViVV @ViVV @ViVV @. \\
1 @>>> N @>i>> G @>{\phi}>> {B} @>>> 1\\
@. @VVV @VVV @VVV @. \\
1 @>>> 1 @>i>> {G/(N Z_G(N))}  @>{\overline{\phi}}>> {\mathcal{M}} @>>> 1\\
@. @VVV @VVV @VVV @. \\
@. 1 @. 1 @. 1 @. 
\end{CD}
\]
from which the lemma follows immediately.
\end{proof}
\begin{corollary}
\label{monofinite}
If the center of $N$ is trivial, then we have
\begin{equation}
\label{monseq}
1\rightarrow N \times Z_G(N) \rightarrow G \rightarrow {\mathcal M}\rightarrow 1.
\end{equation}
In particular, if the monodromy group ${\mathcal M}$ is finite, $G$ is virtually a direct product with $N$ one of the factors, and the index of the direct product subgroup equals the cardinality of ${\mathcal M}.$ To specialize further, if the monodromy is trivial, $G = N \times Z_G(N).$
\end{corollary}

If the center of $N$ is trivial, we also have the following fact:
\begin{theorem}[\cite{kenbrowncoho}[IV.6]]
\label{brownthm}
Extensions of $B$ by $N,$ extensions of the form \eqref{extension} are in one-to-one correspondence to the monodromy representations of $B$ to $\Out(N).$
\end{theorem}
\begin{remark}
\label{morebrown}
If the center of $N$ is nontrivial, on the one hand there is an obstruction to constructing an extension with the given monodromy representation (the obstruction lies in $H^3(B, Z(N))$), and assuming that an extension with a prescribed monodromy representation $\rho: B \rightarrow \Out(N)$ \emph{does} exist, such representations are classified by $H^2(B, Z(N)).$
\end{remark}

Corollary \ref{monofinite} can be used to show a number of rigidity results on group extensions (as usual, using other people's hard work). 

\begin{theorem}
\label{thefarbs}
Let extension \eqref{extension} be such that $N$ is a nonabelian free group or a surface group, and $B$ is a lattice in a semi-simple Lie group of real rank bigger than $1.$ Then the extension is a virtual direct product.
\end{theorem}
\begin{proof} Suppose first that $N$ is a nonabelian free group $F_k.$ The monodromy of the extension is a homomorphic image of $B$ in $\Out(F_k).$ It has been observed by M. Bridson and B. Farb in \cite{bridsonfarb}, using deep work of M. Bestvina, M. Feighn, and M. Handel \cite{titsout}, that a homomorphic image of $B$ in $\Out(F_k)$ is finite, whence the result.

Supose now that $N$ is a fundamental group of a compact surface. B. Farb and H. Masur \cite{farbmasur}, using the very deep results of  V. Kaimanovich and H. Masur \cite{kaimasur}, showed that a homomorphic image of $B$ in $\Out(N)$ is finite, whence the result.
\end{proof}

\begin{theorem}
\label{fujithm}
Let extension \eqref{extension} be such that $N$ is a one-ended centerless word-hyperbolic group, and $B$ is a lattice in a semi-simple lie group of real rank bigger than $1.$ Then, the extension is a virtual direct product.
\end{theorem}
\begin{proof}
It is a theorem of Koji Fujiwara \cite{kojiout} that the homomorphic image of $B$ in $\Out(N)$ is finite, whence the result.
\end{proof}
\begin{remark}
Theorem \ref{fujithm} subsumes the surface case of Theorem \ref{thefarbs}/
\end{remark}

\begin{theorem}
\label{prasadthm}
Let extension \eqref{extension} be such that $N$ is a centerless lattice in a semisimple linear algebraic group of real rank at least two. Then the extension is a virtual direct product.
\end{theorem}
\begin{proof}
It is a theorem of Gopal Prasad \cite{prasadout} that such a lattice has finite outer automorphism group, whence the result.
\end{proof}
\begin{theorem}
\label{paulinthm}
Let extension \eqref{extension} be such that $N$ is a centerless hyperbolic group which has Kazhdan's property T, or a relatively hyperbolic group  that does not split along an elementary   subgroup. Then the extension is a virtual direct product.
\end{theorem}
\begin{proof}
It is a result of F. Paulin \cite{paulinout} that $\Out(N)$ is finite in the first case, and of C. Drutu and M. Sapir \cite{drutusapirout} in the second case.
\end{proof}
\begin{theorem}
\label{gmtthm}
Let extension \eqref{extension} be such that $N$ is the fundamental group of a closed hyperbolic manifold. Then the extension is a virtual direct product.
\end{theorem}
\begin{proof}
It follows from the Mostow Rigidity Theorem and the results of D. Gabai, R. Meyerhoff, and N. Thurston \cite{gmt} that $\Out(N)$ is finite. On the other hand, any torsion-free word-hyperbolic group has trivial center (since if the center were non-trivial, it would contain at least a $\mathbb{Z},$  which, together with a noncentral cyclic subgroup would generate a $\mathbb{Z} \times \mathbb{Z}.$
\end{proof}
\subsection{Monodromy and signature}
\label{monosig}
The results in this section are classical. The first one is due to S. S. Chern, F. Hirzebruch and J-P. Serre:

\begin{theorem}[\cite{chs}]
\label{sigmult}
Let $F \rightarrow E \rightarrow B$ be a fiber bundle such that:
\begin{itemize}
\item The spaces $E, B, F$ are compact oriented manifolds, with compatible orientations.
\item $\pi_1(B)$ acts trivially on $H^*(F).$
\end{itemize}
Then signatures multiply: $\sigma(E) = \sigma(B) \sigma(F).$
\end{theorem}
\begin{corollary}
\label{sigmultcor}
If $F, B$ are surfaces, and the hypotheses of Theorem \ref{sigmult} are satisfied, then $\sigma(E) = 0.$
\end{corollary}
\begin{proof}
The signature of a two-dimensional manifold is defined to be $0.$
\end{proof}
In the below we use the term monodromy in the algebro-geometric sense, that is, we compose the monodromy representation as above with the action on the abelianization of the kernel.
\begin{corollary}[\cite{kot}]
\label{kotcor}
If we have a surface bundle over a surface such that the monodromy action on homology is \emph{finite}, then the signature of the total space is $0,$ that is, $\sigma(E) = 0.$
\end{corollary}
\begin{proof} We pull back the bundle to the finite covering of $B$ corresponding to the kernel of the monodromy representation, and use the multiplicativity of signature.
\end{proof}
\begin{remark} It is a result of Morita \cite{moritabd} that the conclusion of Corollary \ref{kotcor} holds also when the image of the monodromy representation is amenable.
\end{remark}
\bibliographystyle{plain}
\bibliography{CDtest}
\end{document}